\documentclass[12pt,twoside]{article}
\usepackage[latin1]{inputenc}
\usepackage{amssymb,amsmath,amsfonts,paralist,amsrefs}
\usepackage{fancyhdr}

\newtheorem{theorem}{Theorem}

\newtheorem{lemma}[theorem]{Lemma}

\newcommand{\R}{\mathbb{R}}

\newcommand{\Q}{\mathbb{Q}}

\newcommand{\Hy}{\mathbb{H}}

\newcommand{\Ric}{\mbox{Ric}}

\newcommand{\Sp}{\mathbb{S}}

\newcommand{\Lo}{\mathbb{L}}
\DeclareMathAlphabet{\mathpzc}{OT1}{pzc}{m}{it}

\def\<{\langle}
\def\>{\rangle}

\def\bea{\begin{eqnarray*} }
\def\eea{\end{eqnarray*} }
\def\be{\begin{equation} }
\def\ee{\end{equation} }

\def\qed{\ifhmode\unskip\nobreak\fi\ifmmode\ifinner
\else\hskip5 pt \fi\fi\hbox{\hskip5 pt \vrule width4 pt
height6 pt  depth1.5 pt \hskip 1pt }}

\begin{document}

\title{Einstein submanifolds with parallel mean curvature}
\author{Christos-Raent Onti}
\date{}
\maketitle

\begin{abstract} 
We provide a classification of Einstein submanifolds 
in space forms with flat normal bundle and parallel mean curvature. 
This extends a previous result due to Dajczer and Tojeiro \cite{DaRu1} for 
isometric immersions of Riemannian manifolds with constant sectional curvature.
\end{abstract}

\renewcommand{\thefootnote}{\fnsymbol{footnote}} 
\footnotetext{\emph{2010 Mathematics Subject Classification.} Primary 53B25; Secondary 53C40, 53C42.}     
\renewcommand{\thefootnote}{\arabic{footnote}} 

\renewcommand{\thefootnote}{\fnsymbol{footnote}} 
\footnotetext{\emph{Key words and phrases.} Einstein submanifolds, parallel mean curvature, flat normal bundle, principal normals.}     
\renewcommand{\thefootnote}{\arabic{footnote}}

\section{Introduction}
\label{}

The study of isometric immersions of Riemannian manifolds with 
constant sectional curvature into space forms is a basic topic in 
submanifold theory that goes back to Cartan \cite{Ca1,Ca2}. 
Several interesting results towards the classification of these 
immersions have been obtained ever since (see
\cite{Moore, DaRu, DaRu2}). In particular, Dajczer and Tojeiro 
\cite{DaRu1} provided a classification of all such isometric 
immersions with flat normal bundle and parallel mean curvature vector field.

A natural generalization of the concept of Riemannian manifolds with 
constant sectional curvature is the notion of manifolds with constant 
Ricci curvature, namely Einstein manifolds. Fialkow \cite{Fia} and Thomas 
\cite{Tho} initiated the study of isometric immersions of Einstein manifolds 
into space forms. Indeed, after the early work of Fialkow and Thomas, 
Ryan \cite{Ryan} gave a local classification of Einstein hypersurfaces 
in any space form.  
In arbitrary codimension, Di Scala
\cite{discala} proved that Einstein real K\"{a}hler submanifolds of a 
Euclidean space are totally geodesic provided that they are minimal. 
The same conclusion still holds for minimal Einstein submanifolds with flat normal 
bundle in the Euclidean space (see \cite{No}).

In the present paper, we classify isometric
immersions of Einstein manifolds into a 
complete and simply connected Riemannian manifold $\Q_c^N$ 
of constant sectional curvature $c$ of arbitrary codimension, with flat normal 
bundle and parallel mean curvature vector field. Our result, that extends the aforementioned 
result of Dajczer and Tojeiro \cite{DaRu1}, is stated as follows:

\begin{theorem}\label{main}
Let $f\colon M^n\rightarrow \Q_{c}^{N}, n\geq 3,$ be an isometric 
immersion of a connected Einstein manifold with Ricci curvature 
$\lambda$, flat normal bundle and parallel mean curvature vector 
field. Then one of the following holds: 
\begin{enumerate}[(i)]
   \item The immersion $f$ is totally umbilical.
   \item $\lambda=0=c$ and 
            $$
            f(M^n)\subset\Sp^1(r_1)\times\cdots\times \Sp^1(r_k)\times\R^{n-k}\subset\R^{n+k}.
            $$
    \item $\lambda=0<c$ and 
             $$
             f(M^n)\subset\Sp^1(r_1)\times\cdots\times \Sp^1(r_n)\subset\Sp^{2n-1}_c\subset \R^{2n},
             $$ 
             where $r_1^2+\cdots+r_n^2=1/c.$
   \item $\lambda=0>c$ and
            $$
            f(M^n)\subset\Hy^1(r_1) \times \Sp^1(r_2)\times\cdots\times \Sp^1(r_n)\subset 
            \Hy^{2n-1}_c\subset \Lo^{2n},
            $$ where 
            $-r_1^2+r_2^2+\cdots+r_n^2= 1/c.$
    \item $\lambda=c(n-k)>0$ and 
            $$
            f(M^n)\subset\Sp^{m_1}(\rho_1)\times\cdots\times 
            \Sp^{m_k}(\rho_k)\subset\Sp^{n+k-1}_c\subset \R^{n+k},
            $$ 
            where $\rho_i=\sqrt{(m_i-1)/\lambda},\  m_i\geq 2$ for all $1\leq i\leq k$.
   \item $f=j\circ g$, where $g$ is as in (ii), (iii), (iv) or (v) and $j$ is a totally umbilical inclusion.    
\end{enumerate}
\end{theorem}

{\noindent{\it Acknowledgments.}} The author would like to thank Prof. Theodoros 
Vlachos for his valuable suggestions and the referee for his close reading of the 
first draft of this paper which led to various improvements.

\section{Preliminaries}
\label{pre}

Let $f\colon M^n\rightarrow \Q_c^N$ be an isometric immersion of an 
$n$-dimensional Riemannian manifold with Levi-Civita connection 
$\nabla$. The second fundamental form $\alpha$ of $f$ is a symmetric 
section of the vector bundle ${\rm Hom}(TM\times TM,N_f M)$, where 
$N_f M$ is the normal bundle of $f$. We say that $f$ is {\it totally umbilical} 
if 
$$
\alpha(X,Y)=\langle X,Y\rangle H,
$$ 
where $H$ is the mean curvature vector field.

A straightforward computation of the Ricci tensor gives using the 
Gauss equation for $f$ that
\begin{equation}\label{gausseq}
\Ric(X,Y)=c(n-1)\langle X,Y\rangle+n\langle \alpha(X,Y),H\rangle-
\sum_{i=1}^n \langle\alpha(X,X_i),\alpha(Y,X_i)\rangle,\ X,Y\in TM,
\end{equation}
where $X_1,\dots,X_n$ is a local orthonormal frame of the tangent 
bundle $TM$.

The immersion $f$ has {\it flat normal bundle} if the curvature tensor 
of the normal connection $\nabla^\perp$ of $N_fM$ vanishes. 
In this case, it is a standard fact (see \cite{Re1}) that at any point 
$x\in M^n$ there exists a set of unique 
pairwise distinct normal vectors $\eta_i(x)\in N_fM(x),\ 1\leq i\leq s=s(x)$, 
called the {\it principal normals} of $f$ at $x$. 
Moreover, there is an associated orthogonal splitting of the tangent space as
$$
T_xM=E_1(x)\oplus\cdots\oplus E_{s}(x),
$$
where
\begin{equation}\label{E}
E_i(x)=\big\{X\in T_xM:  \alpha(X,Y)=\langle X,Y\rangle\eta_i(x)\ 
\text{for all} \ Y\in T_xM\big\}.
\end{equation}
Hence, the second fundamental form of $f$ at $x$ is given by  
\begin{equation}\label{sff}
\alpha(X,Y)=\sum_{i=1}^{s}\langle X^i,Y^i\rangle \eta_i(x), \ 
X,Y\in T_xM, 
\end{equation} 
where $X^i$ denotes the $E_i(x)$-component of $X$. 

The function $x\in M^n\rightarrow s(x)\in\{1,2,\dots,n\}$
is lower semi-continuous. Hence, if $G_k$ denotes the interior 
of the subset $\{x\in M^n : s(x)=k\},$
then $\cup_{k=1}^n G_k$ is open and dense in $M^n$.
On each $G_k$ the maps $x\in M^n\mapsto \eta_i(x),\ 1\leq i\leq k,$ are smooth 
vector fields called the {\it principal normal vector fields} of $f$ 
and the distributions $x\in M^n\mapsto E_i(x),\ 1\leq i\leq k,$ 
are smooth. The Codazzi equation on $G_k$ is easily seen to yield 
\begin{equation} \label{cod1} 
\langle\nabla_X Y,Z\rangle(\eta_i-\eta_j) = 
\langle X,Y\rangle\nabla_Z^\perp \eta_i 
\end{equation}
and
\begin{equation}\label{cod3}
\langle\nabla_X V, Z\rangle(\eta_j-\eta_l) = 
\langle\nabla_V X, Z\rangle(\eta_j-\eta_i). 
\end{equation}
for all $X,Y\in E_i, Z\in E_j,$ and $V\in E_l$, where 
$1\leq i\neq j\neq l\neq i\leq k$.

\section{Extrinsic product of immersions}
\label{ext}

In this section, we recall the notion of extrinsic product of 
immersions in any space form (cf. \cite{Toj}) that will be 
used in the proof of our main result.

\medskip

A map $f\colon M^n\rightarrow \R^{N}$ from a product manifold 
$M^n=\Pi_{i=1}^k M_i$ is called the {\it extrinsic product of immersions} 
$f_i\colon M_i\rightarrow \R^{m_i},\ 1\leq i\leq k,$ if there exist 
an orthogonal decomposition $\R^N=\Pi_{i=0}^k \R^{m_i}$, with 
$\R^{m_0}$ possibly trivial, such that $f$ is given by 
$$f(x)=(v,f_1(x_1),\dots,f_k(x_k))$$ for all $x=(x_1,\dots,x_k)\in M^n$ 
and $v\in\R^{m_0}$.

A map $f\colon M^n\rightarrow \Sp^{N}(r)\subset\R^{N+1}$ from a 
product manifold $M^n=\Pi_{i=1}^k M_i$ into the sphere 
$$
\Sp^N(r)=\left\{x\in\R^{N+1} : \Vert x\Vert=r\right\},
$$ 
is called the {\it extrinsic product of immersions} 
$f_i\colon M_i\rightarrow \Sp^{m_i-1}(r_i)\subset \R^{m_i},\ 1\leq i\leq k,$
if there exist an orthogonal decomposition 
$\R^{N+1}=\Pi_{i=0}^k \R^{m_i}$, with $\R^{m_0}$ possibly trivial, 
such that $f$ is given by 
$$
f(x)=(v,f_1(x_1),\dots,f_k(x_k))
$$ for all $x=(x_1,\dots,x_k)\in M^n$ with $v\in\R^{m_0}$
and 
$$
\Vert v\Vert^2+\sum_{i=1}^k r_i^2=r^2.
$$

We now consider extrinsic products in the hyperbolic space 
$$\Hy^N(r)=\left\{x=(x_0,\dots,x_{N})\in \Lo^{N+1} : \langle x,x\rangle=-r^2,\ x_0>0\right\},$$
where $\Lo^{N+1}$ denotes the Lorentz space of dimension $N+1$. In this case, there are 
 three different types of extrinsic products called hyperbolic, elliptic 
and parabolic.

A map $f\colon M^n\rightarrow \Hy^N(r)\subset \Lo^{N+1}$ from a product manifold 
$M^n=\Pi_{i=1}^k M_i$ is called the {\it extrinsic product of hyperbolic type} of immersions $f_1,\dots,f_k$ if there exist an orthogonal decomposition
$$\Lo^{N+1}=\Lo^{m_1}\times \Pi_{i=2}^{k+1} \R^{m_i},$$
with $\R^{m_{k+1}}$ possibly trivial, and immersions $$f_1\colon M_1\rightarrow \Hy^{m_1-1}(r_1)\subset \Lo^{m_1}\
\text{and} \ f_i\colon M_i\rightarrow \Sp^{m_i-1}(r_i)\subset \R^{m_i},\ 2\leq i\leq k,$$ 
such that $f$ is given by
$$f(x)=(f_1(x_1),\dots,f_k(x_k),v)$$ 
for all $x=(x_1,\dots,x_k)\in M^n$ with $v\in \R^{m_{k+1}}$ and 
$$-r_1^2+\sum_{i=2}^k r_i^2+\Vert v\Vert^2=-r^2.$$

A map $f\colon M^n\rightarrow \Hy^N(r)\subset \Lo^{N+1}$ from a product manifold $M^n=\Pi_{i=1}^k M_i$ is called
the {\it extrinsic product of elliptic type} of immersions $f_1,\dots,f_k$ if there exist an orthogonal decomposition
$$\Lo^{N+1}=\Pi_{i=1}^k \R^{m_i}\times \Lo^{m_{k+1}},$$ a vector $v\in \Lo^{m_{k+1}}$, and immersions
$f_i\colon M_i\rightarrow \Sp^{m_i-1}(r_i)\subset \R^{m_i},\ 1\leq i\leq k,$
such that $f$ is given by $$f(x)=(f_1(x_1),\dots,f_k(x_k),v)$$ for all $x=(x_1,\dots,x_k)\in M^n$ 
 with $$\sum_{i=1}^k r_i^2+\langle v,v\rangle=-r^2.$$

Finally, a map $f\colon M^n\rightarrow \Hy^N(r)\subset \Lo^{N+1}$ from a product manifold $M^n=\Pi_{i=1}^k M_i$ is called
the {\it extrinsic product of parabolic type} of immersions $f_1,\dots,f_k$ if there exist $s\in\{1,\dots,k\}$, an orthogonal decomposition
$$\Lo^{N+1}=\Lo^{l+1}\times \Pi_{i=s+1}^{k+1} \R^{m_i},$$ with $\R^{m_{k+1}}$ possibly trivial, and immersions
$$f_i\colon M_i\rightarrow \R^{m_i},\ 1\leq i\leq s,\ \text{and}\ f_j\colon M_j\rightarrow \Sp^{m_j-1}(r_j)\subset\R^{m_j},\ s+1\leq j\leq k \ \text{if}\ s<k,$$
such that $f$ is given by
$$f(x)=\big({\mathsf i}(f_1(x_1),\dots,f_s(x_s)),f_{s+1}(x_{s+1}),\dots,f_{k}(x_{k}),v\big)$$
for all $x=(x_1,\dots,x_k)\in M^n$ with $v\in\R^{m_{k+1}}$. Here $${\mathsf i}\colon \Pi_{i=1}^s \R^{m_i}=\R^{l-1} \rightarrow \Hy^l(r_1)\subset\Lo^{l+1}$$ denotes an umbilical inclusion with $$-r_1^2+\sum_{i=2}^k r_i^2+\Vert v\Vert^2=-r^2.$$

\medskip

Let $f\colon M^n \to\Q_c^N$ be an isometric immersion of a
Riemannian manifold. If $M^n=\Pi_{i=1}^k M_i$ is a product manifold 
then the second fundamental form $\alpha$ is  said to be {\it adapted} to the product structure of $M^n$ if 
$$\alpha(X_i,X_j)=0\ \text{for all} \ X_i\in TM_i, \ X_j\in TM_j\ \text{with}\ 1\leq i\neq j\leq k,$$ 
where the tangent bundles $TM_i$ are identified with the corresponding tangent distributions
to $M^n$.
The next result, which is due to Moore \cite{Mo3} for the case $c=0$ and to Molzan \cite{Mo,Re} 
for the case $c\neq 0$, shows that products of isometric immersions are characterized by this property among 
isometric immersions of Riemannian products.

\begin{theorem}\label{MMR}
Let $f\colon M^n\rightarrow \Q_{c}^{N}$ be an isometric immersion of a Riemannian product manifold $M^n=\Pi_{i=1}^k M_i$ 
with adapted second fundamental form. Then $f$ is an extrinsic product of isometric immersions.
\end{theorem}

\section{The proof}
\label{proof1}

Let $f\colon M^n\to \Q_c^N$ be an isometric immersion as in Theorem \ref{main}. In the following 
we are working on an open subset $G_k$ with $k\geq 2$.

\begin{lemma}\label{lem}
Around every point $p\in G_k$ there is a neighborhood $U$ that is a Riemannian product of Riemannian manifolds 
$M_1,\dots,M_k$. Moreover, $\left. f\right\vert_U$ is the extrinsic product of totally umbilical isometric immersions 
$f_1,\dots,f_k$.
\end{lemma}

\begin{proof}
We claim that each distribution $E_i,\ 1\leq i\leq k,$ is parallel, that is 
$$\nabla_X Y\in E_i\ \text{for all}\ X\in TG_k,\ Y\in E_i \ \text{and}\ 1\leq i\leq k.$$
First we prove that $\nabla_X Y\in E_i$ for all $X,Y\in E_i$ and $1\leq i\leq k$. Indeed, from \eqref{cod1} we have that
 $$\langle \nabla_X Y,Z\rangle(\eta_i-\eta_j)=\langle X,Y\rangle\nabla_Z^\perp \big(\eta_i-\frac{n}{2}H\big)$$ 
 for any $X, Y\in E_i$ and $Z\in E_j$ with $j\neq i$.
Thus, we obtain
\begin{equation}\label{pp}
2\langle \nabla_X Y,Z\rangle \big<\eta_i-\eta_j,\eta_i-\frac{n}{2}H\big>= \langle X,Y\rangle Z\big(\Vert \eta_i-\frac{n}{2}H\Vert^2\big).
\end{equation}
Using \eqref{gausseq}, \eqref{sff} and the hypothesis that $M^n$ is an Einstein manifold we have 
\begin{equation}\label{en}
\big\Vert \eta_i-\frac{n}{2}H\big\Vert^2=\big\Vert \frac{n}{2}H\big\Vert^2-\lambda+c(n-1),\ 1\leq i\leq k.
\end{equation}
Using \eqref{en} we observe that
$$ \big<\eta_i-\eta_j,\eta_i-\frac{n}{2}H\big>=\big\Vert \eta_i-\frac{n}{2}H\big\Vert^2-\big<\eta_i-\frac{n}{2}H,\eta_j-\frac{n}{2}H\big>\neq 0.$$
Then \eqref{pp} implies that $\nabla_X Y\in E_i$ for all $X,Y\in E_i$ and any $1\leq i\leq k$.

Now, we show that $\nabla_X Y\in E_i$ for all $X\in E_j$ and $Y\in E_i$ with $j\neq i$. To this aim, we consider 
$Y\in E_i,\ X\in E_j,\ Z\in E_l\subset E_i^\perp$ 
and distinguish the following two cases. 

If $l=j$, then by using the previous argument, we have
\begin{equation}\label{par1}
\langle \nabla_X Y,Z\rangle=-\langle Y,\nabla_X Z\rangle=0.
\end{equation}

If $l\neq j$, then from \eqref{cod3} we obtain
$$\langle \nabla_X Y,Z\rangle (\eta_l-\eta_i)=\langle \nabla_Y X,Z\rangle (\eta_l-\eta_j).$$
We claim that $\eta_l-\eta_i$ and $\eta_l-\eta_j$ are pointwise linearly independent. Indeed, we assume to the contrary 
 that $$\eta_l-\eta_i=\mu (\eta_l-\eta_j)$$ for some $\mu\in\R\smallsetminus\{0\}$. Then 
 $$(1-\mu)\big(\eta_l-\frac{n}{2}H\big)=\eta_i-\frac{n}{2}H-\mu\big(\eta_j-\frac{n}{2}H\big).$$
By taking the norms and using \eqref{en} we obtain that $\eta_i=\eta_j$, which is a contradiction.
Thus 
\begin{equation}\label{par2}
\langle \nabla_X Y,Z\rangle=0 \ \text{for all}\ X\in E_j, Y\in E_i, Z\in E_l, 
\end{equation}
with $l\neq i\neq j$.
Therefore, from \eqref{par1} and \eqref{par2}, we obtain that $\nabla_X Y\in E_i$ for all $X\in E_j$ and $Y\in E_i$ with $j\neq i$. 
This completes the proof of our claim. 

Now, de Rham's theorem implies that around every point $p\in G_k$ there is a neighborhood
$U$ that is the Riemannian product of the integral manifolds 
$M_1,\dots,M_k$ of the distributions $E_1,\dots,E_k$ respectively, through a point $q\in U$. 
Moreover, since the second fundamental form of $f$ is adapted, 
Theorem \ref{MMR} implies that $\left. f\right\vert_U$ is an extrinsic product of 
isometric immersions $f_1,\dots,f_k$, which due to \eqref{E} have to be totally umbilical. 
\qed
\end{proof}

\begin{lemma}\label{lem2}
Let $m_i=\dim M_i$. Then the following holds:
\begin{equation}\label{eq}
(m_i-1) (c+\Vert\eta_i\Vert^2)=\lambda,\ 1\leq i\leq k.
\end{equation}
Moreover, if $m_i\geq 2$ then the sectional 
curvature of $M_i$ is 
\begin{equation}\label{sec}
K_{M_i}=\frac{\lambda}{m_i-1}.
\end{equation}
Furthermore, if $m_i\geq 2$ for all 
$1\leq i\leq k$, then $\lambda>0$. 
\end{lemma}

\begin{proof}
It follows from Lemma \ref{lem} and the Gauss equation
that the principal normals $\eta_1,\dots,\eta_k$ of $f$ satisfy 
\begin{equation}\label{gauss1}
\langle\eta_i,\eta_j\rangle=-c,\ 1\leq i\neq j\leq k.
\end{equation}
Equation \eqref{eq} follows from our assumption, \eqref{gausseq} 
and \eqref{gauss1}. If $m_i\geq 2$ then  \eqref{eq}
and the Gauss equation imply \eqref{sec}.

Now, suppose that $m_i\geq 2$ for all $1\leq i\leq k$ and assume to 
the contrary that $\lambda\leq 0$. Then  
 \eqref{eq} implies that 
 
\begin{equation}\label{ee}
\Vert \eta_i\Vert^2\leq -c\  \ \text{for all}\ \ 1\leq i\leq k
\end{equation}
 and thus $c<0$. Therefore,  
from \eqref{gauss1}, \eqref{ee} and the Cauchy-Schwarz inequality we obtain that
\begin{equation}\label{eq2}
-c=\<\eta_i,\eta_j\>\leq \Vert \eta_i\Vert \Vert \eta_j\Vert\leq -c\ \ \text{for all}\ \  i\neq j.
\end{equation}
Hence, 
$\eta_j=\mu_{ij}\eta_i$ for some $\mu_{ij}>0$ and $1\leq i\neq j\leq k.$
From \eqref{gauss1}
 and \eqref{ee} we have that $\mu_{ij}\geq 1$ for all $1\leq i\neq j\leq k.$
Since $\eta_j=\mu_{ij}\eta_i=\mu_{ij}\mu_{ji}\eta_j$, it follows that $\mu_{ij}\leq 1$. Therefore, 
$\mu_{ij}=1$ which is a contradiction.\qed
\end{proof}

\begin{lemma}\label{lem3}
Assume that there exists $1\leq i\leq k$ such that $m_i=1$. Then $U$ is flat.
\end{lemma}

\begin{proof}
The proof follows from Lemmas \ref{lem} and \ref{lem2}. \qed
\end{proof}

\vspace{2ex}

\noindent \emph{Proof of Theorem \ref{main}:}
Let $f\colon M^n\to \Q_c^N,\ n\geq 3,$ be an isometric immersion 
of a connected Einstein manifold with Ricci curvature $\lambda$, flat normal 
bundle and parallel mean curvature vector field. 
If $c\neq 0$, we always view $\Q_c^N$ as an umbilical 
hypersurface of the Euclidean space $\R^{N+1}$
or the Lorentzian space $\Lo^{N+1}$ according to the sign of $c$.
In the following we are working on an open subset $G_k$.

If $k=1$, then $f$ is a totally umbilical immersion. 
In the sequel we assume that $k\geq 2$ and let $p\in G_k$. Then, according to Lemma 
\ref{lem} we have that there exists a neighborhood $U$ of $p$ that is a Riemannian 
product of Riemannian manifolds $M_1,\dots, M_k$ and 
$\left. f\right\vert_U$ is an extrinsic product of totally umbilical isometric immersions 
$f_1,\dots,f_k$.

We distinguish two cases.

If there exists $1\leq i\leq k$ such that $m_i=1$ then the 
result follows from Lemma \ref{lem3} and Theorem 1 of \cite{DaRu1}. 

We now assume that $m_i\geq 2$ for all $1\leq i\leq k$.
If $c=0$, then each $f_i$ is an umbilical isometric immersion into
$\R^{m_i+1}$, where $\R^N=\R^m\times\Pi_{i=1}^k \R^{m_i+1}$. 
Therefore, bearing in mind \eqref{sec} we obtain that
$$
f(U)\subset\Sp^{m_1}(\rho_1)\times\cdots\times 
\Sp^{m_k}(\rho_k)\subset\R^N.
$$

If $c>0$, then each $f_i$ is an umbilical isometric 
immersion into $\Sp^{m_i+1}(r_i)\subset \R^{m_i+2}$, where 
$\R^{N+1}=\R^m\times\Pi_{i=1}^k \R^{m_i+2}$. Thus, 
by using \eqref{sec} we obtain that
$$
f(U)\subset\Sp^{m_1}(\rho_1)\times\cdots\times 
\Sp^{m_k}(\rho_k)\subset\Sp_c^N.
$$

Finally, if $c<0$ then $f$ is the extrinsic product of umbilical 
isometric immersions $f_1,\dots,f_k$ of either hyperbolic, elliptic or parabolic type. 

If $f_1,\dots,f_k$ is of hyperbolic type, then $f_1$ is an umbilical isometric immersion
into $\Hy^{m_1+1}(r_1)\subset \Lo^{m_1+2}$ and each $f_i$ is an umbilical isometric immersion into 
$\Sp^{m_i+1}(r_i)\subset\R^{m_i+2},$ $2\leq i\leq k$, where
$\Lo^{N+1}=\Lo^{m_1+2}\times\Pi_{i=2}^k \R^{m_i+2}\times\R^{m_{k+1}}$. Using \eqref{sec}
we obtain that 
$$
f(U)\subset\Sp^{m_1}(\rho_1)\times\cdots\times 
\Sp^{m_k}(\rho_k)\subset\Hy_c^N.
$$
Clearly, $f(U)$ is contained in a totally umbilical submanifold of $\Hy_c^N$ of positive sectional
curvature.

If $f_1,\dots,f_k$ is of elliptic type then each $f_i$ is an umbilical isometric immersion into 
$\Sp^{m_i+1}(r_i)\subset\R^{m_i+2},$ where 
$\Lo^{N+1}=\Pi_{i=1}^{k} \R^{m_i+2}\times \Lo^{m}$. 
Bearing in mind \eqref{sec} we obtain that 
$$
f(U)\subset\Sp^{m_1}(\rho_1)\times\cdots\times 
\Sp^{m_k}(\rho_k)\subset\Hy_c^N.
$$
Clearly, $f(U)$ is contained in a flat totally umbilical submanifold of $\Hy_c^N$.

If $f_1,\dots,f_k$ is of parabolic type then there exist $s\leq k$
such that each $f_i$ is an umbilical isometric immersion into $\R^{m_i+1}$
for $1\leq i\leq s$ and each $f_j$ is an umbilical isometric immersion into
$\Sp^{m_j+1}(r_j)\subset\R^{m_j+2}$ for $s+1\leq j\leq k$ if $s<k$, where 
$\Lo^{N+1}=\Lo^{l}\times\Pi_{i=s+1}^k \R^{m_i+2}\times \R^{m_{k+1}}$.
Thus, from \eqref{sec} we obtain that 
$$
f(U)\subset{\sf i}\big(\Pi_{i=1}^s\Sp^{m_i}(\rho_i)\big) 
\times \Pi_{i=s+1}^k\Sp^{m_i}(\rho_i)\subset\Hy_c^N.
$$
Again in this case $f(U)$ is contained in a flat totally umbilical submanifold of $\Hy_c^N$.

Finally, since $M^n$ is connected and the above different type of submanifolds cannot be smoothly attached  
we have that $f(M^n)$ is an open subset of one of the above and this completes the 
proof. \qed

\bigskip

{\indent{\small\textsc{Department of Mathematics, University of Ioannina, 45110 Ioannina, Greece}}} \\
{\indent{\small {\it E-mail address}: \texttt{chonti@cc.uoi.gr}}}

\end{document}